\documentclass{amsart}

\usepackage[dvipdfmx]{graphicx}
\allowdisplaybreaks

\usepackage{diagbox}

\theoremstyle{plain}
\newtheorem{theorem}{Theorem}[section]
\newtheorem{proposition}[theorem]{Proposition}
\newtheorem{lemma}[theorem]{Lemma}

\theoremstyle{definition}
\newtheorem*{definition}{Definition}
\newtheorem{example}[theorem]{Example}

\theoremstyle{remark}
\newtheorem{remark}[theorem]{Remark}

\begin{document}

\title{Coloring links by the symmetric group of degree three}

\author{Kazuhiro Ichihara}
\address{Kazuhiro Ichihara \\ Department of Mathematics \\ College of Humanities and Sciences, Nihon University \\  3-25-40 Sakurajosui, Setagaya-ku, Tokyo 156-8550, Japan}
\email{ichihara.kazuhiro@nihon-u.ac.jp}

\author{Eri Matsudo}
\address{Eri Matsudo \\ The Institute of Natural Sciences \\ Nihon University \\ 3-25-40 Sakurajosui, Setagaya-ku, Tokyo 156-8550, Japan}
\email{matsudo.eri@nihon-u.ac.jp}

\thanks{The first author is partially supported by Grant-in-Aids for Scientific Research (C) (No. 18K03287), The Ministry of Education, Culture, Sports, Science and Technology, Japan.}


\begin{abstract}
We consider the number of colors for colorings of links by the symmetric group $S_3$ of degree $3$. 
For knots, such a coloring corresponds to a Fox 3-coloring, and thus the number of colors must be 1 or 3. 
However, for links, there are colorings by $S_3$ with 4 or 5 colors. 
In this paper, we show that if a 2-bridge link admits a coloring by $S_3$ with 5 colors, then the link also admits such a coloring with only 4 colors. 
\end{abstract}

\maketitle

\section{Introduction}

One of the most well-known invariants of knots and links would be the Fox 3-coloring, originally introduced by R.~Fox. 
For example, it is described in \cite[Chap. VI, Exercises, 6, pp.92--93]{CrowellFox1963}. 
In this exercise, readers are asked to show that a knot has a diagram which is 3-colorable if and only if the knot group can be mapped homomorphically onto the symmetric group of degree 3. 
In view of this, as a generalization of the Fox 3-coloring, we consider the colorings of links by the symmetric group of degree $3$, which we denote by $S_3$. 

\begin{definition}
Let $D$ be a diagram of a link. 
We call a map $\Gamma:\{ \mbox{arcs of $D$} \} \rightarrow S_3 \setminus \{e\}$ is called an {\it $S_3$-coloring} on $D$ if it satisfies $\Gamma(x)\Gamma(y)=\Gamma(z)\Gamma(x)$ (respectively, $\Gamma(x)\Gamma(z)=\Gamma(y)\Gamma(x)$) at a positive (resp.~negative) crossing on $D$, 
where $x$ denotes the over arc, $y$ and $z$ the under arcs at the crossing supposing 
$y$ is the under arc before passing through the crossing and $z$ is the other. 
\end{definition}

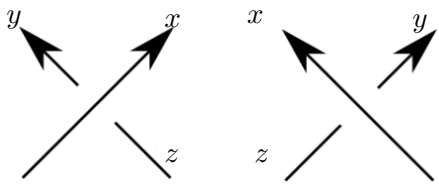
\begin{figure}[htbt]
\centering
  {\unitlength=1mm
  \begin{picture}(100,20)
   \put(25,0){\includegraphics[width=.45\textwidth]{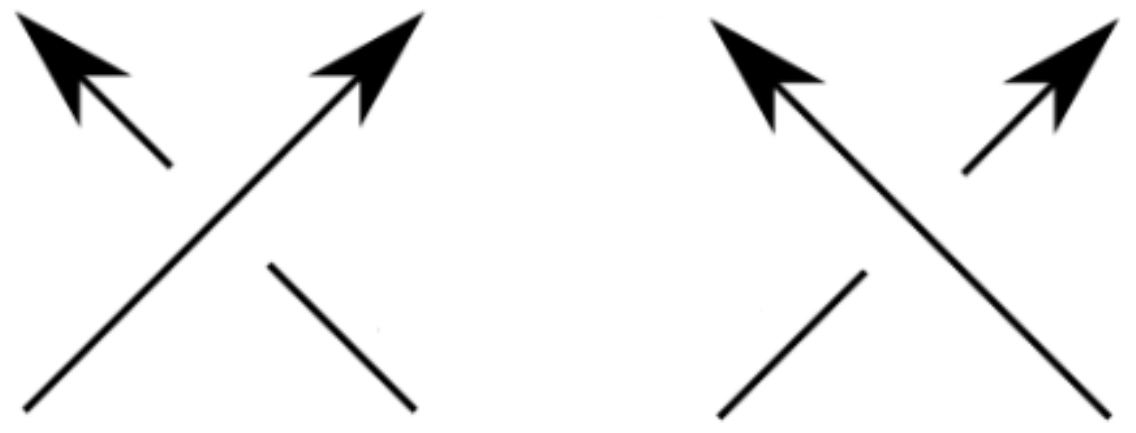}}
   \put(45,21){{\large $x$}}
   \put(78,21){{\large $y$}}
   \put(45,3){{\large $z$}}
   \put(57,3){{\large $z$}}
   \put(24,21.5){{\large $y$}}
   \put(56,21.5){{\large $x$}}
  \end{picture}}
\caption{Crossing conditions for $S_3$-coloring}\label{Fig/condition}
\label{Fig2}
\end{figure}

The image $\Gamma(a)$ of an arc $a$ on $D$ by an $S_3$-coloring $\Gamma$ is said to be a {\it color} on $a$ with respect to $\Gamma$. 

Note that an $S_3$-coloring on a diagram $D$ of a link $L$ gives a representation $G_L \to S_3$ of the link group $G_L = \pi_1(S^3-L)$ of $L$, and conversely, a representation of $G_L$ to $S_3$ gives an $S_3$-coloring on any diagram $D$ of a link $L$. 

Actually, for knots, such an $S_3$-coloring corresponds to a Fox 3-coloring, as stated in \cite[Chap. VI, Exercises, 6, pp.92--93]{CrowellFox1963}. 
Thus the number of colors for such colorings must be 1 or 3. 
However, for links, there exist colorings by $S_3$ with 4 or 5 colors. 
See the example below. (See the next section for details.) 

\begin{figure}[htb]
\centering
\includegraphics[width=\textwidth]{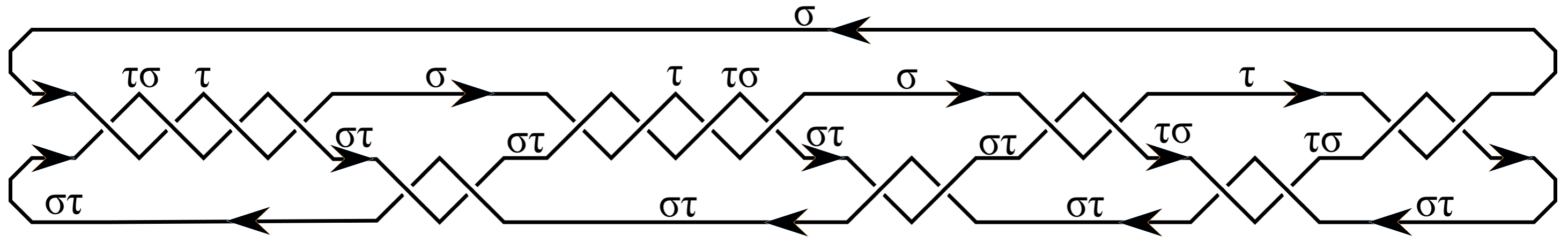}
\caption{a link diagram with an $(S_3,4)$-coloring}\label{2bridge-ex}
\end{figure}

Focusing the number of colors, in this paper, we call an $S_3$-coloring $\Gamma$ an \textit{$(S_3, n)$-coloring} if $\Gamma$ uses $n$ colors for an integer $n \in \{ 1,2,3,4,5 \}$. 
An $(S_3, 1)$-coloring is said to be a \textit{trivial} $S_3$-coloring. 
A link $L$ is said to be \textit{$S_3$-colorable} (resp. \textit{$(S_3,n)$-colorable}) if $L$ has a diagram which admits a non-trivial $S_3$-coloring (resp. an $(S_3,n)$-coloring). 
Then, for links, the following holds. 

\begin{proposition}\label{prop1}
Any $(S_3,4)$-colorable link is also $(S_3,5)$-colorable. 
Precisely, if a link $L$ has a diagram which admits an $S_3$-coloring with 4 colors, then $L$ also has another diagram which admits an $S_3$-coloring with 5 colors. 
\end{proposition}

On the other hand, one can ask if the converse does hold: Is an $(S_3,5)$-colorable link always $(S_3,4)$-colorable? 
It seems to expect too much naively, but there are some results on the Fox coloring related to this question. 
For example, it is known that if a knot $K$ is Fox 5-colorable, then $K$ has a diagram which admits a Fox 5-coloring with only 4 colors \cite{Satoh}. 
Also the second author \cite{Matsudo} and independently M. Zhang, X. Jin\ and\ Q. Deng \cite{ZhangJinDeng} proved that if a link $L$ is $\mathbb{Z}$-colorable, then $L$ has a diagram which admits a $\mathbb{Z}$-coloring with only 4 colors. 

About the question above, in this paper, we obtain the following for  $2$-bridge links. 

\begin{theorem}\label{main}
Any $(S_3,5)$-colorable $2$-bridge link $L$ is $(S_3,4)$-colorable. 
\end{theorem}

In the next section, we describe the local behavior of $S_3$-colorings on links preparing lemmas. 
Then, in Section 3, we give a proof of Theorem~\ref{main}. 
By Theorem~\ref{main}, all the $(S_3, 5)$-colorable 2-bridge links are $(S_3, 4)$-colorable. 
Some of them actually are also $(S_3, 3)$-colorable, but some others are not. 
In the last section, among 2-bridge links, we determine the double twist links and the torus links that are $(S_3, 4)$-colorable but not $(S_3, 3)$-colorable.

\section{Local behavior of $S_3$-colorings}

Throughout the paper, we set a presentation of $S_3$ as $\langle \sigma, \tau \mid \sigma^2= \tau^2 = e , \sigma\tau\sigma=\tau\sigma\tau \rangle$, where $e$ denotes the identity element of $S_3$. 
Then, note that $S_3 = \{ e, \sigma,\tau,\sigma\tau\sigma,\sigma\tau,\tau\sigma \}$ as a set. 

In this section, we observe the local behavior of $S_3$-colorings on links, and prepare some lemmas used in the remaining sections. 

Let $L$ be a link with a diagram $D$. 
Suppose that $D$ admits a non-trivial $S_3$-coloring $\Gamma$. 
At a crossing of $D$, let $x$ denote the over arc, $y$ and $z$ the under arcs at the crossing supposing 
$y$ is the under arc before passing through the crossing and $z$ is the other. 
See Figure~\ref{Fig/condition}. 
Then the possible colors of the arcs $x$, $y$, $z$ assigned by $\Gamma$ can be summarized in the following table. 

\begin{table}[htb]
\caption{Colors on $y$ when the colors on $x$ and $z$ are assigned.}\label{table1}
\begin{tabular}{|c||c|c|c|c|c|}
\hline
\diagbox{Color on $x$}{Color on $z$} & $\sigma$ & $\tau$ & $\sigma\tau\sigma$ & $\sigma\tau$ & $\tau\sigma$ \\
\hline\hline
$\sigma$ & $\sigma$ & $\sigma\tau\sigma$ & $\tau$ & $\tau\sigma$ & $\sigma\tau$ \\
\hline
$\tau$ & $\sigma\tau\sigma$ & $\tau$ & $\sigma$ & $\tau\sigma$ & $\sigma\tau$ \\
\hline
$\sigma\tau\sigma$ & $\tau$ & $\sigma$ & $\sigma\tau\sigma$ & $\tau\sigma$ & $\sigma\tau$ \\
\hline
$\sigma\tau$ & {\small \diagbox[dir=NE]{$\sigma\tau\sigma$}{$\tau$} } & {\small \diagbox[dir=NE]{$\sigma$}{$\sigma\tau\sigma$} } & {\small \diagbox[dir=NE]{$\tau$}{$\sigma$} } & $\sigma\tau$ & $\tau\sigma$ \\
\hline
$\tau\sigma$ & {\small \diagbox[dir=NE]{$\tau$}{$\sigma\tau\sigma$} } & {\small \diagbox[dir=NE]{$\sigma\tau\sigma$}{$\sigma$} } & {\small \diagbox[dir=NE]{$\sigma$}{$\tau$} } & $\sigma\tau$ & $\tau\sigma$ \\
\hline
\end{tabular}
\end{table}

In the table above, 
\begin{tabular}{|c|}
\hline
{\small \diagbox[dir=NE]{$\alpha$}{$\beta$} }\\
\hline
\end{tabular}
means that the color on $y$ is $\alpha$ (resp. $\beta$ ) if the crossing is positive (resp. negative). 

\begin{remark}
Also, from Table~\ref{table1}, we see that any link with at least 2 components admits an $S_3$-coroling with 2 colors $\{ \sigma\tau, \tau\sigma\}$. 
See Figure~\ref{2-col} for example. 

\begin{figure}[htb]
\centering
\includegraphics[width=.4\textwidth]{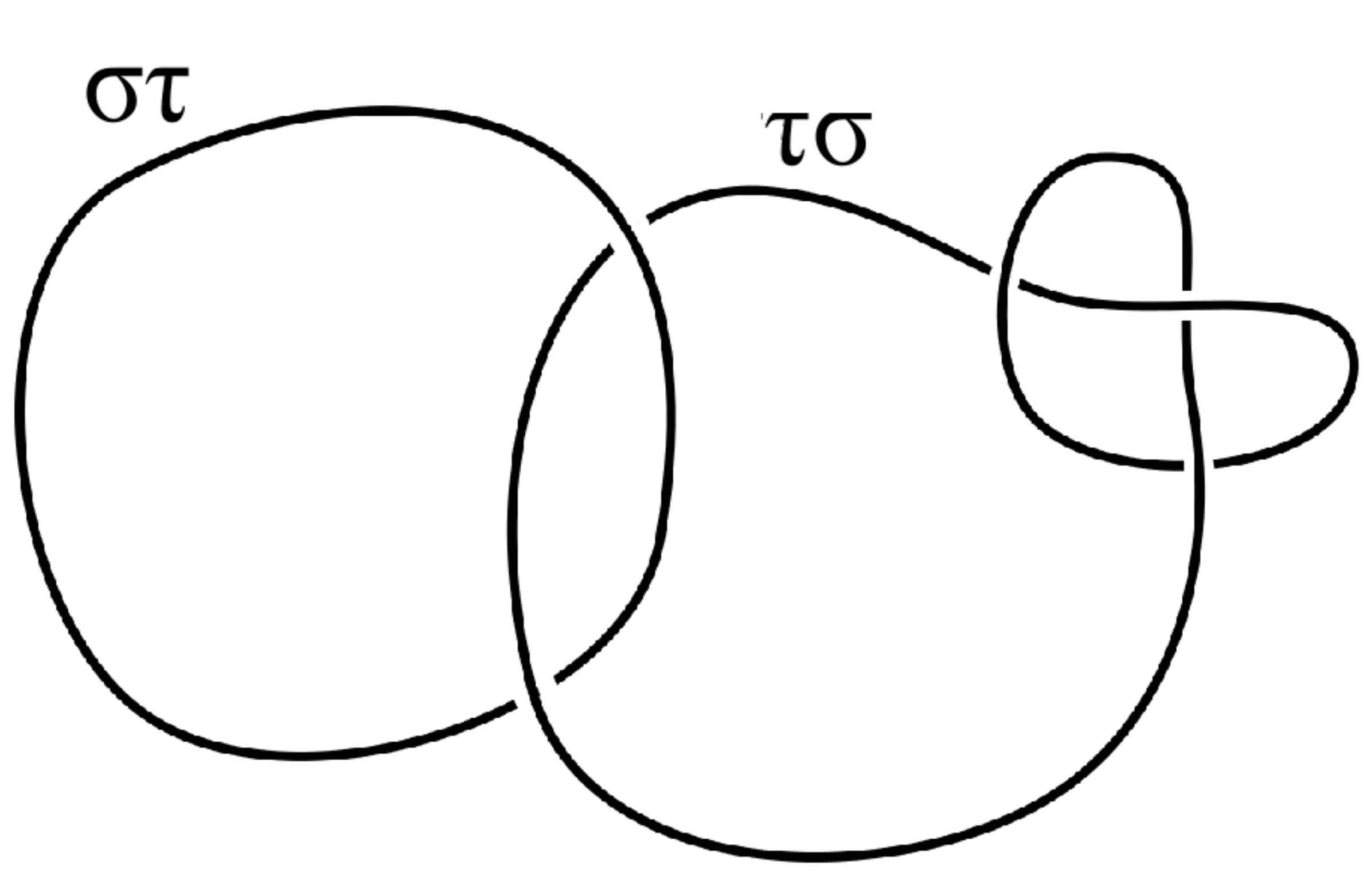}
\caption{An $S_3$-colored link with 2 colors}\label{2-col}
\end{figure}
\end{remark}

The next is our fundamental lemma, which we will use implicitly and repeatedly. 
It follows from Table~\ref{table1}. 

\begin{lemma}\label{lem1}
Let $\Gamma$ be an $S_3$-coloring on a diagram $D$ of a link $L$. 
Then the set of the colors on arcs of $D$ corresponding to one component of $L$ are either a subset of $\{ \sigma, \tau, \sigma \tau \sigma\}$ or a subset of $\{ \sigma\tau , \tau\sigma\}$. 
The set $\{ \sigma, \tau, \sigma \tau \sigma\}$ or $\{ \sigma\tau , \tau\sigma\}$ for a component of $L$ is unchanged by modifying the diagram and the coloring by Reidemeister moves. 
\end{lemma}

\begin{proof}
From Table~\ref{table1}, if one of the under arcs at a crossing of a link diagram is colored by one of $\{ \sigma, \tau, \sigma \tau \sigma\}$ or one of $\{ \sigma\tau , \tau\sigma\}$ by an $S_3$-coloring, then the other under arc is also. 
Thus the first statement holds. 
One can check the local behavior of $S_3$-colorings by Reidemeister moves to keep the set of colors on the related arcs. 
This implies the second statement. 
\end{proof}

We remark that this lemma can be derived from considering the conjugacy classes of $S_3$ or the conjugate quandle structure of $S_3$. 

For a diagram $D$ of a knot,  there is a one-to-one correspondence between a non-trivial Fox $3$-coloring on a diagram $D$ and an $(S_3,3)$-coloring on $D$ as follows. 

\begin{lemma}\label{lem2}
(i) For a non-splittable $(S_3,3)$-colorable link $L$, the set of colors for an $(S_3,3)$-coloring on a diagram of $L$ is $\{ \sigma, \tau, \sigma \tau \sigma\}$. 
(ii) For a knot $K$,  there is a one-to-one correspondence between a Fox $3$-coloring on a diagram $D$ and an $S_3$-coloring on $D$ of $K$. 
Thus a knot $K$ is $S_3$-colorable if and only if $K$ is Fox 3-colorable. 
In particular, if a knot is $(S_3,n)$-colorable, then $n = 1 \mbox{ or } 3$. 
\end{lemma}

\begin{proof}
(i) Suppose that a diagram $D$ of a link $L$ admits an $(S_3,3)$-coloring $\Gamma$. 
From Lemma~\ref{lem1}, the set of colors on each component of the link are either of $\{ \sigma, \tau, \sigma \tau \sigma\}$ or $\{ \sigma\tau , \tau\sigma\}$. 
Let $\alpha,\beta,\gamma$ be the three colors used by $\Gamma$. 
If $\alpha\in\{\sigma,\tau,\sigma\tau\sigma\}$ and $\beta,\gamma\in\{\sigma\tau,\tau\sigma\}$, then, by Table~\ref{table1}, the arc colored by $\alpha$ is constantly an over arc, or an under arc at the crossing with the over arc colored by $\alpha$, a contradiction. 
Thus the component with an arc colored by $\alpha$ is splittable from the other components, implying that $L$ is splittable. 
Similarly, the same argument applies for the case $\alpha\in\{\sigma\tau,\tau\sigma\}$ and $\beta,\gamma\in\{\sigma,\tau,\sigma\tau\sigma\}$. 
Thus, if $L$ is non-splittable, the set of 3 colors for $\Gamma$ is $\{ \sigma, \tau, \sigma \tau \sigma\}$. 

\noindent
(ii) Suppose that $K$ is Fox 3-colorable, i.e., $K$ has a diagram $D$ of a knot $K$ admits a non-trivial $S_3$-coloring. 
Then, by Lemma~\ref{lem1}, the set of colors appearing are either from $\{ \sigma, \tau, \sigma \tau \sigma\}$ or from $\{ \sigma\tau , \tau\sigma\}$. 
If an arc on $D$ could have a color from $\{ \sigma\tau , \tau\sigma\}$, since $K$ has only one component, then, by Table~\ref{table1}, the coloring uses only one color on $D$, that is, the coloring is trivial, contradicting $\Gamma$ is non-trivial. 
It follows that the set of colors for the coloring must be from $\{ \sigma, \tau, \sigma \tau \sigma\}$. 
In this case, seeing Table~\ref{table1},  we note that all the 3 colors appear or only single color appears at each of the crossings of $D$. 
Thus, if the coloring $\Gamma$ is non-trivial, then $\Gamma$ must use 3 colors. 
By replacing the colors $\{ \sigma, \tau, \sigma \tau \sigma \}$ to $\{ 0 , 1, 2 \}$, we can verify by Table~\ref{table1} that a Fox $3$-coloring on $D$  can be obtained from $\Gamma$. 
Conversely, one can obtain an $S_3$-coloring from a Fox $3$-coloring on a knot diagram by setting the colors $\{ 0 , 1, 2 \}$ to the colors $\{ \sigma, \tau, \sigma \tau \sigma \}$. 
See Table~\ref{table1} again. 
\end{proof}

\begin{remark}
For splittable links, the lemma above does not hold. 
See Figure~\ref{split-3col} for example. 
\end{remark}

\begin{figure}[htb]
\centering
\includegraphics[width=.4\textwidth]{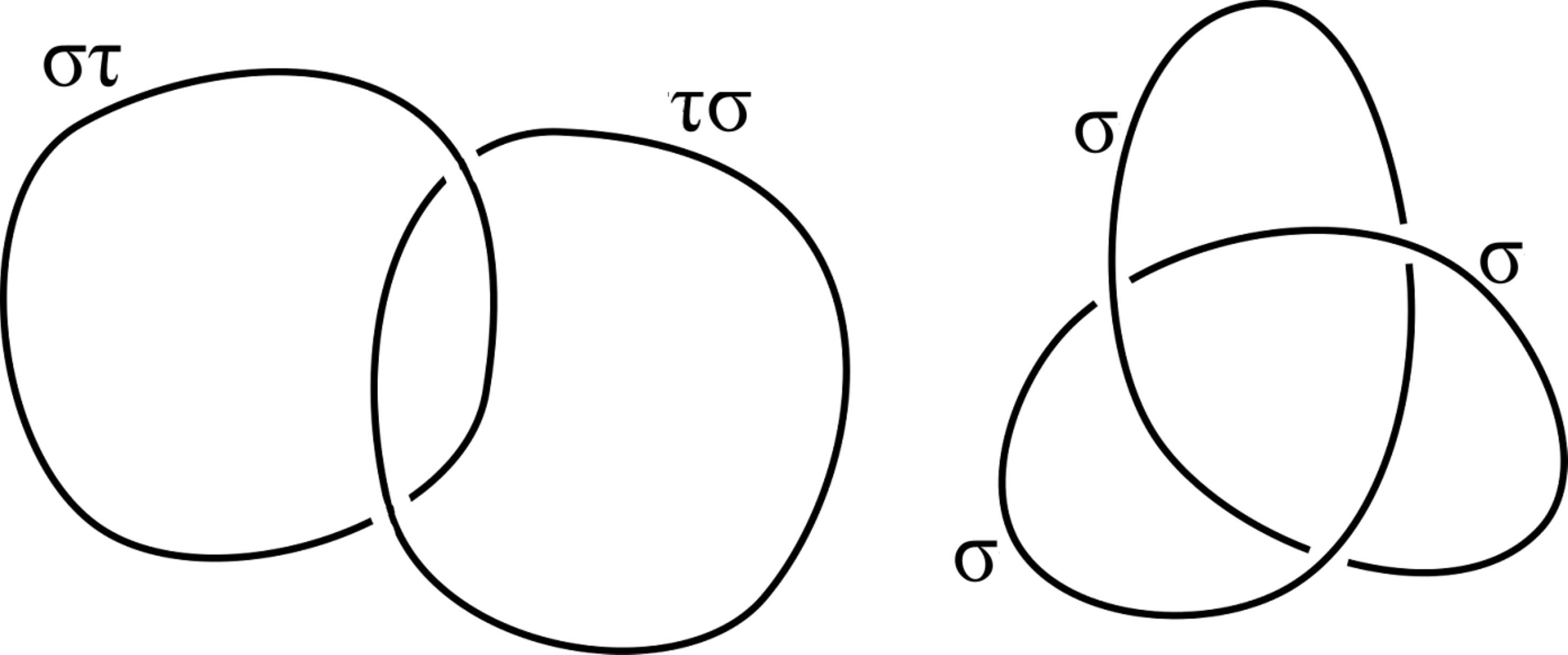}
\caption{An $(S_3,3)$-coloring on a splittable 2-component link with the three colors $\{ \sigma, \sigma\tau, \tau\sigma \}$}\label{split-3col}
\end{figure}

From Lemma~\ref{lem2}, a knot $K$ is $S_3$-colorable if and only if $K$ is Fox 3-colorable. 
In particular, if a knot is $(S_3,n)$-colorable, then $n= 1 \mbox { or } 3$. 

On the other hand, if a link $L$ has at least 2 components, then $L$ can be $(S_3,n)$-colorable with $n \ge4$, as illustlated in Figure~\ref{2bridge-ex} for an example. 

For such $S_3$-colorings with 4 or 5 colors, 
we have the following. 

\begin{lemma}\label{lem}
Let $L$ be a non-splittable link and $D$ a diagram of $L$.
Suppose that $D$ admits an $(S_3,4)$-coloring or an $(S_3,5)$-coloring, say $\Gamma$.
(i) The set of colors of \,$\Gamma$ contains at least $2$ colors from $\{\sigma,\tau,\sigma\tau\sigma\}$ and $2$ colors from $\{\sigma\tau,\tau\sigma\}$. 
(ii) The $S_3$-coloring induced from $\Gamma$ on a diagram of $L$ obtained by Reidemeister moves from $D$ has at least $4$ colors. 
\end{lemma}

\begin{proof}
Suppose that $D$ admits an $(S_3,4)$-coloring or an $(S_3,5)$-coloring, say $\Gamma$. 
Since $L$ has at least two components by Lemma~\ref{lem2}(ii), one of which is colored by $\Gamma$ with $\{\sigma,\tau,\sigma\tau\sigma\}$, and the other is colored by $\Gamma$ with $\{\sigma\tau,\tau\sigma\}$ by Lemma~\ref{lem1}. 

\noindent
(i) 
Suppose for a contradiction that $\Gamma$ uses only one color, say $\gamma$, from $\{\sigma\tau,\tau\sigma\}$. 
Then, by Table~\ref{table1}, the arc colored by $\gamma$ is constantly an over arc, or an under arc at the crossing with the over arc colored by $\gamma$. 
This means that the component is splittable, and it contradicts that $L$ is non-splittable. 
Thus the set of colors of $\Gamma$ contains at least $2$ colors from $\{\sigma,\tau,\sigma\tau\sigma\}$ and $2$ colors from $\{\sigma\tau,\tau\sigma\}$. 

\noindent
(ii)
Let $\Gamma'$ be the $S_3$-coloring induced from $\Gamma$ on a diagram of $L$ obtained by Reidemeister moves from $D$. 
Then, by Lemma~\ref{lem1}, such sets of colors on the components are unchanged by Reidemeister moves, and so, $\Gamma'$ has at least one color in $\{\sigma,\tau,\sigma\tau\sigma\}$ and one color in $\{\sigma\tau,\tau\sigma\}$. 
Moreover, since $L$ is non-splittable, there exists at least one crossing where the pair of the colors above appear. 
Then, by Table~\ref{table1}, there has to be one more color at the crossing. 
Thus $\Gamma'$ uses at least 3 colors with one color in $\{\sigma,\tau,\sigma\tau\sigma\}$ and one color in $\{\sigma\tau,\tau\sigma\}$. 
It follows from Lemma~\ref{lem2}, together with above, the coloring $\Gamma'$ is not an $(S_3,3)$-coloring. 
Therefore, if $D$ admits an $(S_3,4)$-coloring or an $(S_3,5)$-coloring, then any $S_3$-coloring on a diagram of $L$ obtained by Reidemeister moves from $D$ with the coloring has at least $4$ colors. 
\end{proof}

Now we give a proof of Proposition~\ref{prop1}. 

\begin{proof}[Proof of Proposition~\ref{prop1}]
Let $L$ be an $(S_3,4)$-colorable link and $D$ a diagram of $L$ with an $(S_3,4)$-coloring $\Gamma$. 

If $L$ is non-splittable, then there exist $2$ colors in $\{\sigma,\tau,\sigma\tau\sigma\}$ and $2$ colors in $\{\sigma\tau,\tau\sigma\}$ on $D$ from Lemma~\ref{lem} (i).
Let $\alpha\in\{\sigma,\tau,\sigma\tau\sigma\}$ be the color which $\Gamma$ does not use. 
Consider an arc on $D$ colored by $\beta,\gamma\in\{\sigma,\tau,\sigma\tau\sigma\}$ with $\beta,\gamma\neq\alpha$. 
Then one can deform $D$ and $\Gamma$ to a diagram with a coloring so that $\alpha$ appears by using Reidemeister move II repeatedly, as illustrated in Figure~\ref{figprop1}. 
Then the coloring so obtained uses five colors by Lemma~\ref{lem} (ii).

\begin{figure}[htb]
\centering
\includegraphics[width=.5\textwidth]{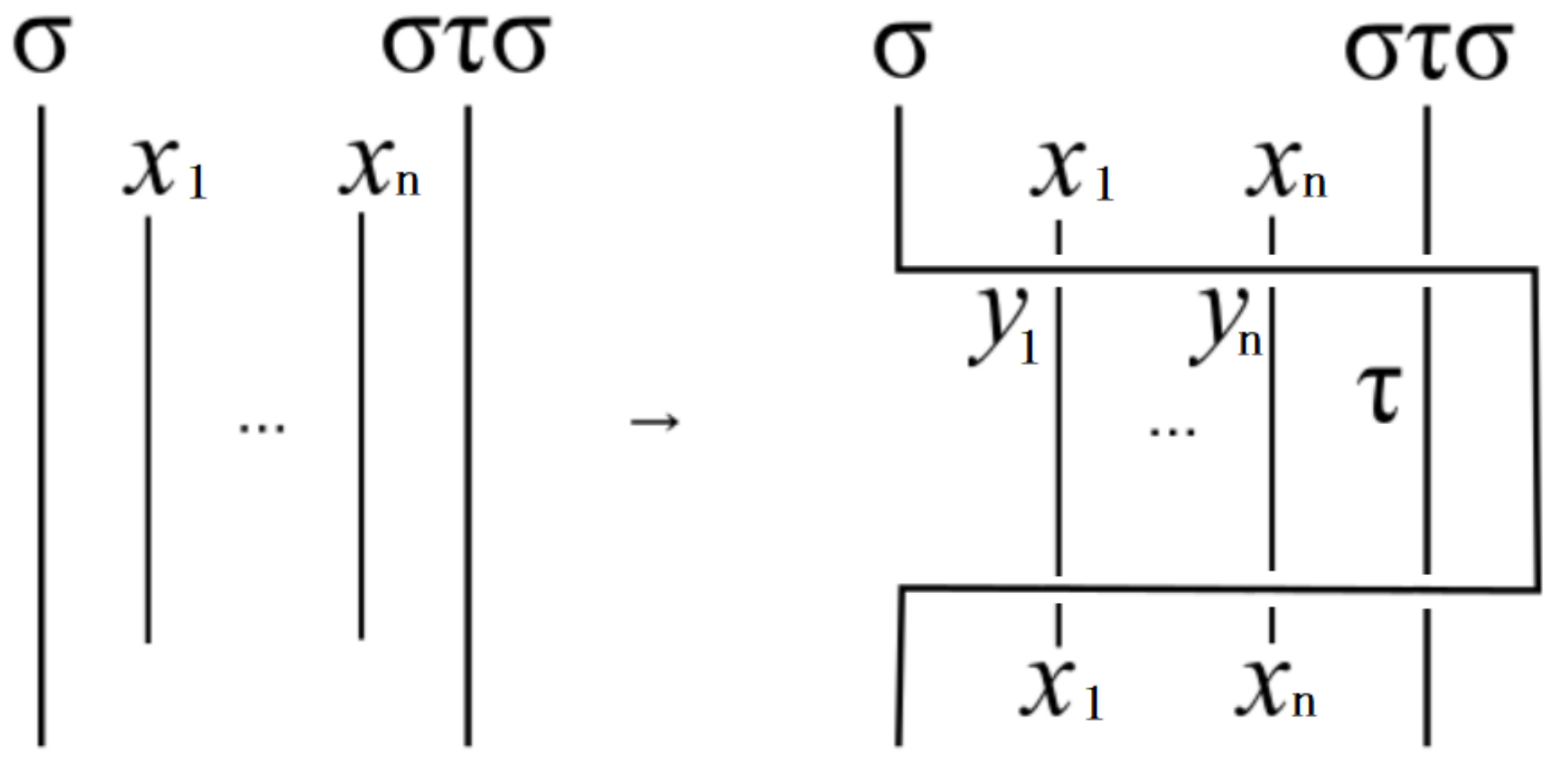}
\caption{Making $\tau$ appear from $\{\sigma,\sigma\tau\sigma,\sigma\tau,\tau\sigma\}$}
\label{figprop1}
\end{figure}

When $L$ is splittable, we also have to consider the case that there exists $3$ colors in $\{\sigma,\tau,\sigma\tau\sigma\}$ and $1$ color in $\{\sigma\tau,\tau\sigma\}$ on $D$. 
In this case, let $\alpha\in\{\sigma\tau,\tau\sigma\}$ be a color which $\Gamma$ does not use. 
On the other hand, $D$ contains an arc colored by $\beta\in\{\sigma\tau,\tau\sigma\}$ with $\beta\neq\alpha$. 
Then one can deform $D$ with the coloring to a diagram with a coloring with $\alpha$ by using Reidemeister move II repeatedly, as illustrated in Figure~\ref{prop2}.
Then the coloring so obtained uses five colors by Lemma~\ref{lem} (ii).

\begin{figure}[htb]
\centering
\includegraphics[width=.5\textwidth]{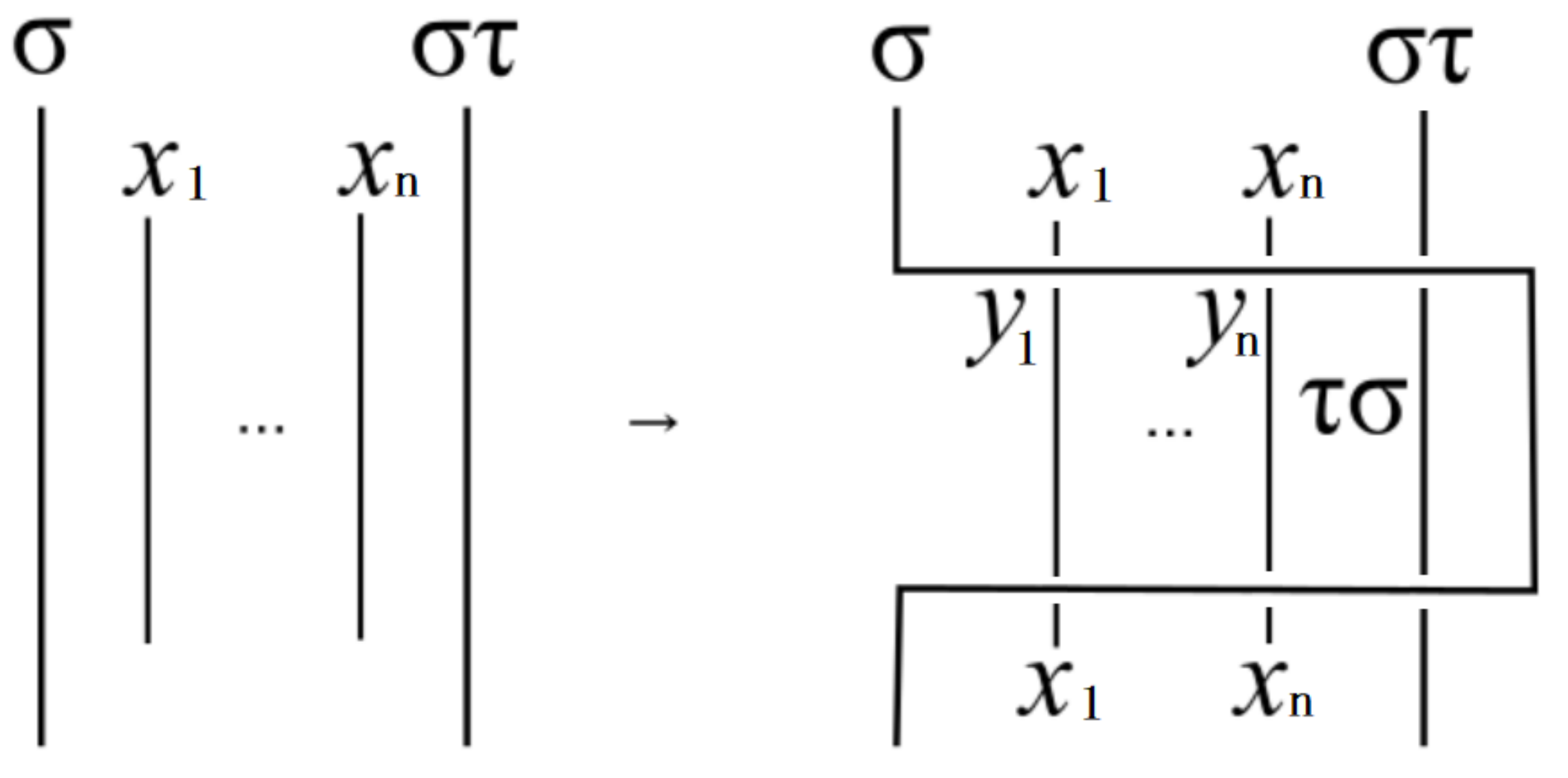}
\caption{Making $\tau\sigma$ appear from $\{\sigma,\tau,\sigma\tau\sigma,\sigma\tau\}$}
\label{prop2}
\end{figure}

\end{proof}

\section{Proof of Theorem~\ref{main}}

In this section, we give a proof of Theorem~\ref{main}. 
Recall that it is known that a 2-bridge link always has a Conway diagram $C(2a_1, 2b_1,\dots, 2b_m, 2a_{m+1})$ depicted in Figure~\ref{general}. 
See \cite[Chapter 2]{Kawauchi} about the 2-bridge links and the Conway diagrams (called "Conway's normal form" in the book) for example. 
In the follwing, we always assume that $a_i \ne 0$ and $b_j\ne 0$. 

\begin{figure}[htb]
\centering
\includegraphics[width=.6\textwidth]{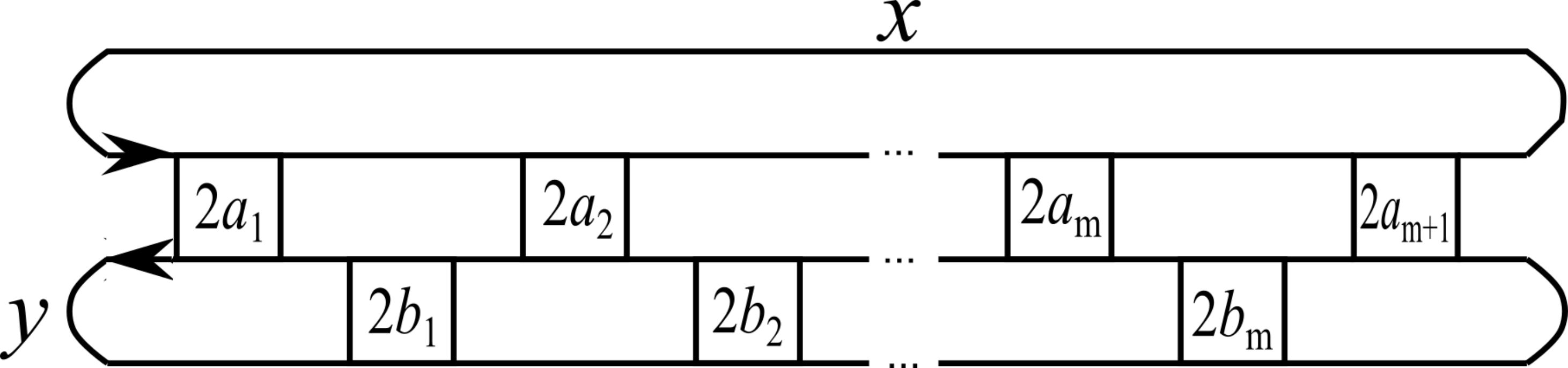}
\caption{a Conway diagram $C(2a_1, 2b_1, \dots, 2b_m, 2a_{m+1})$}
\label{general}
\end{figure}

We first show the following lemma. 

\begin{lemma}\label{lem:conway}
The Conway diagram $C(2a_1, 2b_1, 2a_2, 2b_2,\dots, 2b_m, 2a_{m+1})$ of a 2-bridge link $L$ admits an $(S_3,4)$-coloring if $\sum_{i=1}^{m+1}|a_i|\equiv 0\pmod 2$ holds for the diagram. 
\end{lemma}

Note that the last congruent equation is equivalent to that the linking number of the two components of a two-bridge link is even. 

\begin{proof}[Proof of Lemma~\ref{lem:conway}]
We try to construct an $(S_3,4)$-coloring on a Conway diagram $C(2a_1, 2b_1, 2a_2, 2b_2,\dots, 2b_m, 2a_{m+1})$ from the left end of the diagram. 

We fix colors on arcs $x,y$ in Figure~\ref{general} as $\sigma$ and $\sigma\tau$ respectively. 
Then, let us try to make a coloring by setting the color on the arc next to the right of a colored arc by using Table~\ref{table1}.
Repeatedly perform this procedure from left to right.

First we see the colors in the twist regions with $2 a_i$ crossings $( 1 \le i \le m+1)$. 
Since $2a_i$ is even,  pairs of colors at before and after $2a_i$ crossings are the same or another color pair. 
Precisely, if $a_i$ is even, the pairs of colors before and after $2a_i$ crossings are coincide. 
If $a_i$ is odd, the pairs of colors before and after $2a_i$ crossings are distinct, but in a fixed pattern. 
For example, if a pair of colors $\{\sigma,\sigma\tau\}$ appears before the twist, then the pairs of colors on the parallel arcs during the twist are $\{\sigma,\sigma\tau\}$ or $\{\tau,\tau\sigma\}$ alternately as illustrated in Figure \ref{pr2}. 
In particular, during the twists, only 4 colors can appear. 

\begin{figure}[htb]
\centering
\includegraphics[width=.6\textwidth]{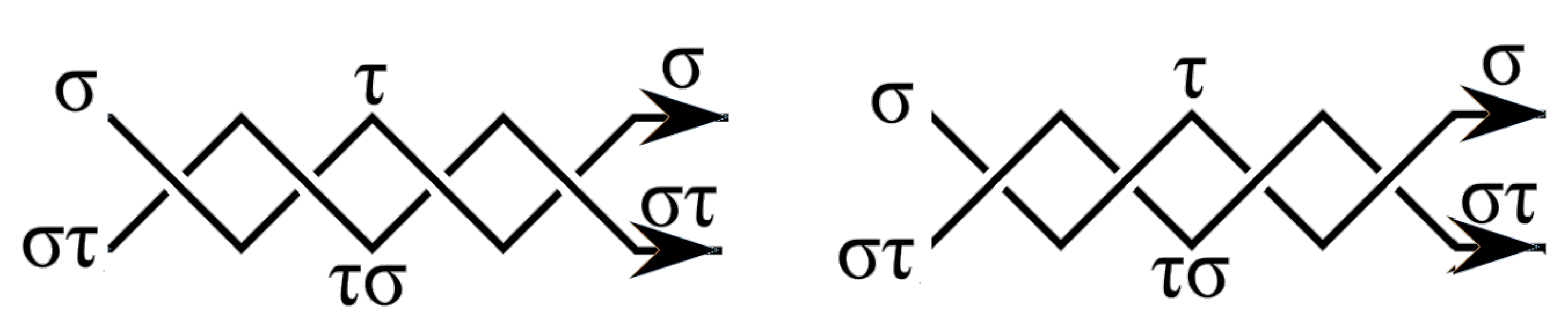}
\caption{colors in the twist with $2a_i$ crossings}
\label{pr2}
\end{figure}

Next we see the colors in the twist regions with $2 b_j$ crossings $( 1 \le j \le m)$. 
On the arcs in the twist with $2 b_j$ crossings, just two colors $\sigma\tau,\tau\sigma$ appear. 
Moreover, the colors on the parallel arcs before and after the twisting are the same. 
See Figure \ref{pr1}.
\begin{figure}[htb]
\centering
\includegraphics[width=.2\textwidth]{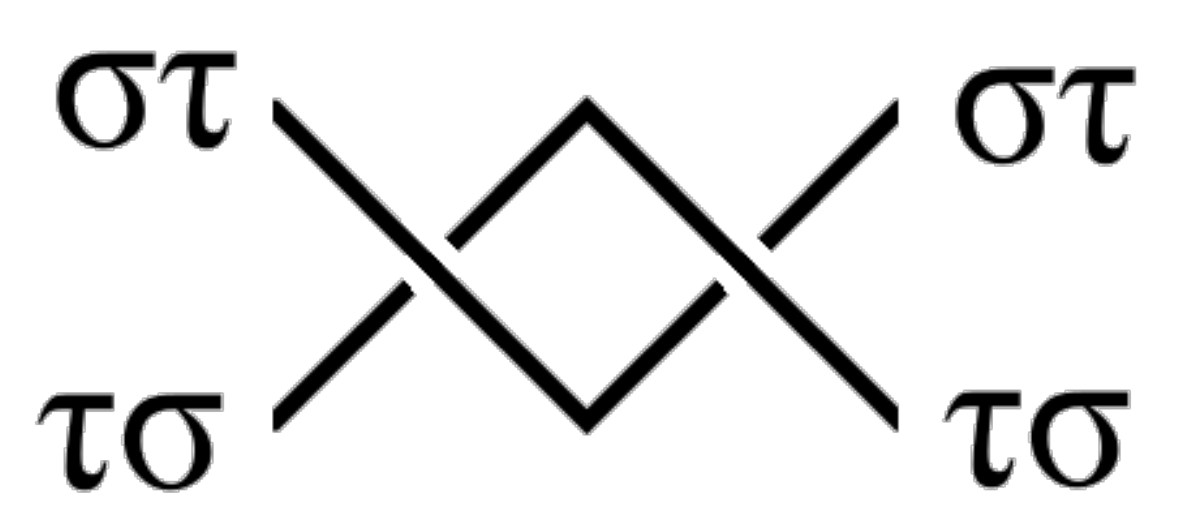}
\caption{colors in the twist with $2b_i$ crossings}
\label{pr1}
\end{figure}

From this procedure, checking the right-end of the diagram, we can obtain an $S_3$-coloring on the diagram if and only if $\sum_{i=1}^{m+1}|a_i|\equiv 0\pmod 2$ holds. 
By the construction, the coloring so obtained uses only 4 colors. 
\end{proof}

\begin{proof}[Proof of Theorem~\ref{main}]
Let $L$ be an $(S_3,5)$-colorable $2$-bridge link. 
By Reidemeister moves, we deform a diagram $D$ of $L$ with an $(S_3,5)$-coloring to a Conway diagram $D_C=C(2a_1, 2b_1, \dots, 2b_m, 2a_{m+1})$ as shown in Figure~\ref{general} with the induced $S_3$-coloring $\Gamma$. 
By Lemma~\ref{lem} (i) and (ii), the coloring $\Gamma$ uses at least $2$ colors from $\{\sigma,\tau,\sigma\tau\sigma\}$ and $2$ colors from $\{\sigma\tau,\tau\sigma\}$. 
Moreover, by Lemma~\ref{lem1}, the arcs contained in one component have the colors either from $\{\sigma,\tau,\sigma\tau\sigma\}$ or from $\{\sigma\tau,\tau\sigma\}$. 

Now we consider the colors on the arcs $x$ and $y$ in Figure~\ref{general} by $\Gamma$. 

When $\Gamma(x)\in\{\sigma,\tau,\sigma\tau\sigma\}$ and $\Gamma(y)\in\{\sigma\tau,\tau\sigma\}$, then, by retaking the colors if necessary, the coloring is completely the same as that constructed in the proof of Lemma~\ref{lem:conway}. 
That is, $\Gamma$ is an $(S_3,4)$-coloring on the diagram, and $\sum_{i=1}^{m+1}|a_i|\equiv 0\pmod 2$ must hold. 

Consider the case that $\Gamma(x)\in\{\sigma\tau,\tau\sigma\}$ and $\Gamma(y)\in\{\sigma,\tau,\sigma\tau\sigma\}$. 
Then one can deform the diagram and the coloring to $\Gamma'$ so that $\Gamma'(x) \in \{\sigma,\tau,\sigma\tau\sigma\}$ and $\Gamma'(y)\in\{\sigma\tau,\tau\sigma\}$ by Reidemeister moves. 
Precisely, it is achieved by rotating the interior part of the thin line inside-out, keeping the exterior part of the line fixed as illustrated in Figure \ref{pr3}. 
Also see \cite[Chapter 9]{Murasugi}. 

\begin{figure}[htb]
\centering
\includegraphics[width=.6\textwidth]{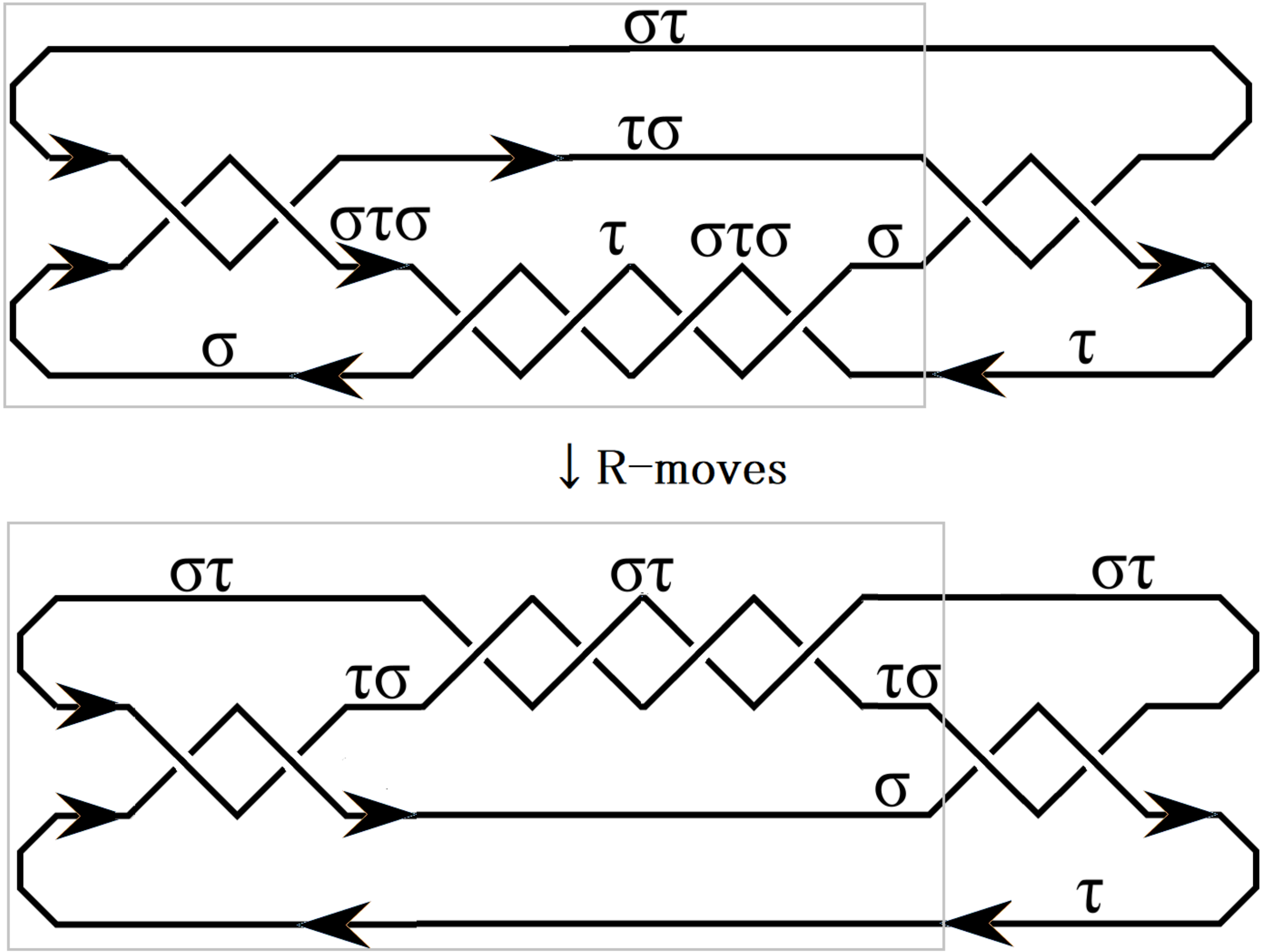}
\caption{Reidemeister moves to $\Gamma'(x)\in\{\sigma,\tau,\sigma\tau\sigma\}$}
\label{pr3}
\end{figure}

After such modifications, in the same way as the proof of Lemma~\ref{lem:conway}, we see that the condition $\displaystyle \Sigma_{i=1}^{m+1}|a_i|\equiv 0\pmod 2$ have to be satisfied, and $\Gamma'$ is an $(S_3,4)$-coloring on the diagram. 

It concludes that if a 2-bridge link $L$ is $(S_3,5)$-colorable, then $L$ has a Conway diagram $C(2a_1, 2b_1, 2a_2, 2b_2,\dots, 2b_m, 2a_{m+1})$ satisfying $\sum_{i=1}^{m+1}|a_i|\equiv 0\pmod 2$, and the diagram admits an $(S_3,4)$-coloring, i.e., the link $L$ is $(S_3,4)$-colorable.
This completes the proof of Theorem~\ref{main}.
\end{proof}

\section{Examples}

From Theorem \ref{main}, any $(S_3,5)$-colorable $2$-bridge link is $(S_3,4)$-colorable. 
Among such $(S_3,4)$-colorable links, there exists some of the links which is also $(S_3,3)$-colorable and the others are not. 
In this section, we collect some examples of $S_3$-colorings for 2-bridge links, and in particular, consider double twist links. 
One of the simplest 2-bridge links would be 2-bridge torus links, that are the torus links with only two strands. 

\begin{example}[The torus link $T(2,q)$]
By Theorem~\ref{main}, the torus link $T(2,q)$ is $(S_3,4)$-colorable if and only if $q\equiv 0 \pmod {4}$. 
The next figure depicts a torus link with $(S_3,4)$-coloring which is not $(S_3,3)$-colorable. 

\begin{figure}[htb]
\centering
\includegraphics[width=.4\textwidth]{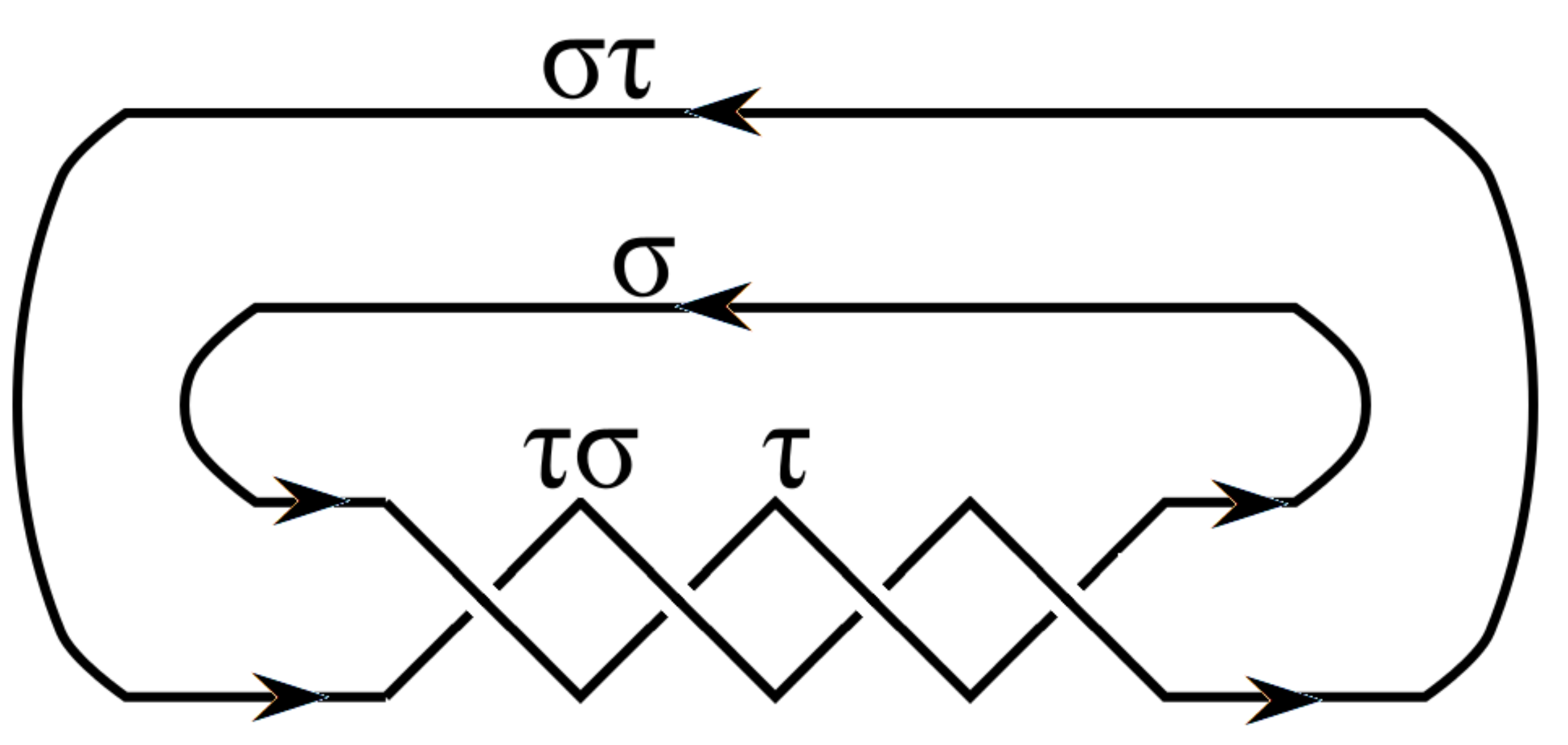}
\caption{Torus link $T(2,4)$ with an $(S_3,4)$-coloring}\label{toruslink}
\end{figure}

In fact, by using Table~\ref{table1}, one can see that the standard torus diagram of $T(2,q)$ (Figure~\ref{toruslink}) is $(S_3,4)$-colorable if and only $q\equiv 0 \pmod {4}$ and $T(2,q)$ is $(S_3,3)$-colorable if and only if $q\equiv 0 \pmod {3}$. 
See Figures~\ref{pr2} and ~\ref{3twists}.

\begin{figure}[htb]
\centering
\includegraphics[width=.4\textwidth]{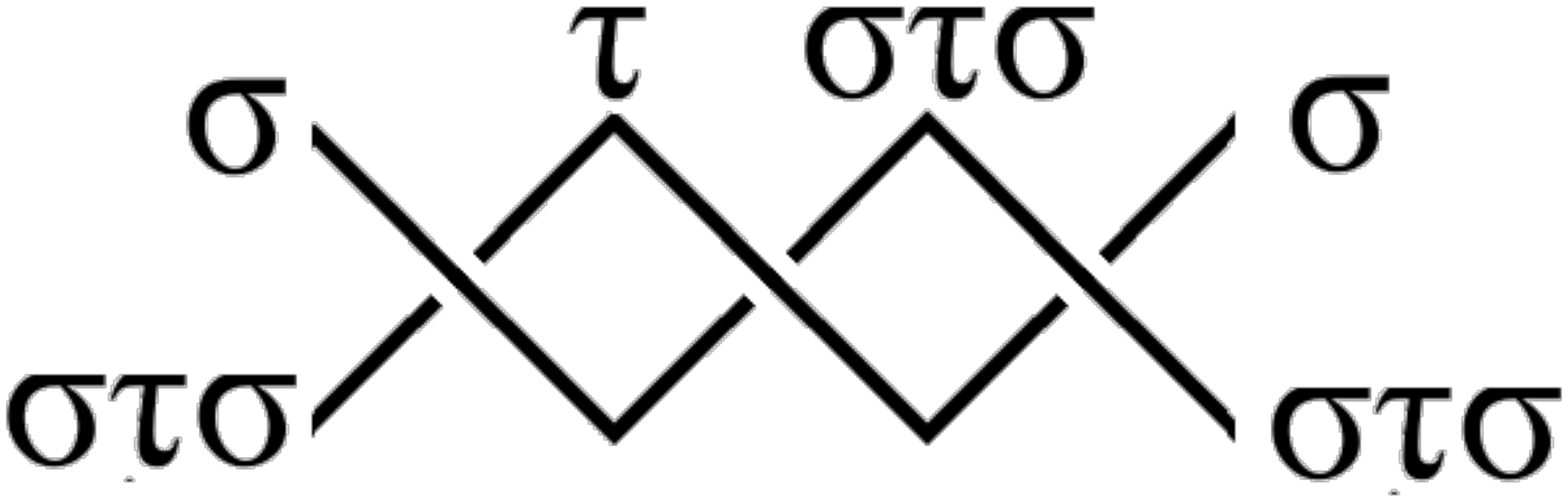}
\caption{Twists with $\{\sigma,\tau,\sigma\tau\sigma\}$}\label{3twists}
\end{figure}

Also, by \cite{KauffmanLopes}, the determinant of $T(2,q)$ is $q$, and so, $T(2,q)$ is Fox 3-colorable, equivalently, is $(S_3,3)$-colorable if and only if $q\not\equiv 0\pmod 3$. 

For example, the torus link $T(2,12)$ is $(S_3,n)$-colorable for $n=3,4,5$. 

\begin{figure}[htb]
\centering
\includegraphics[width=.8\textwidth]{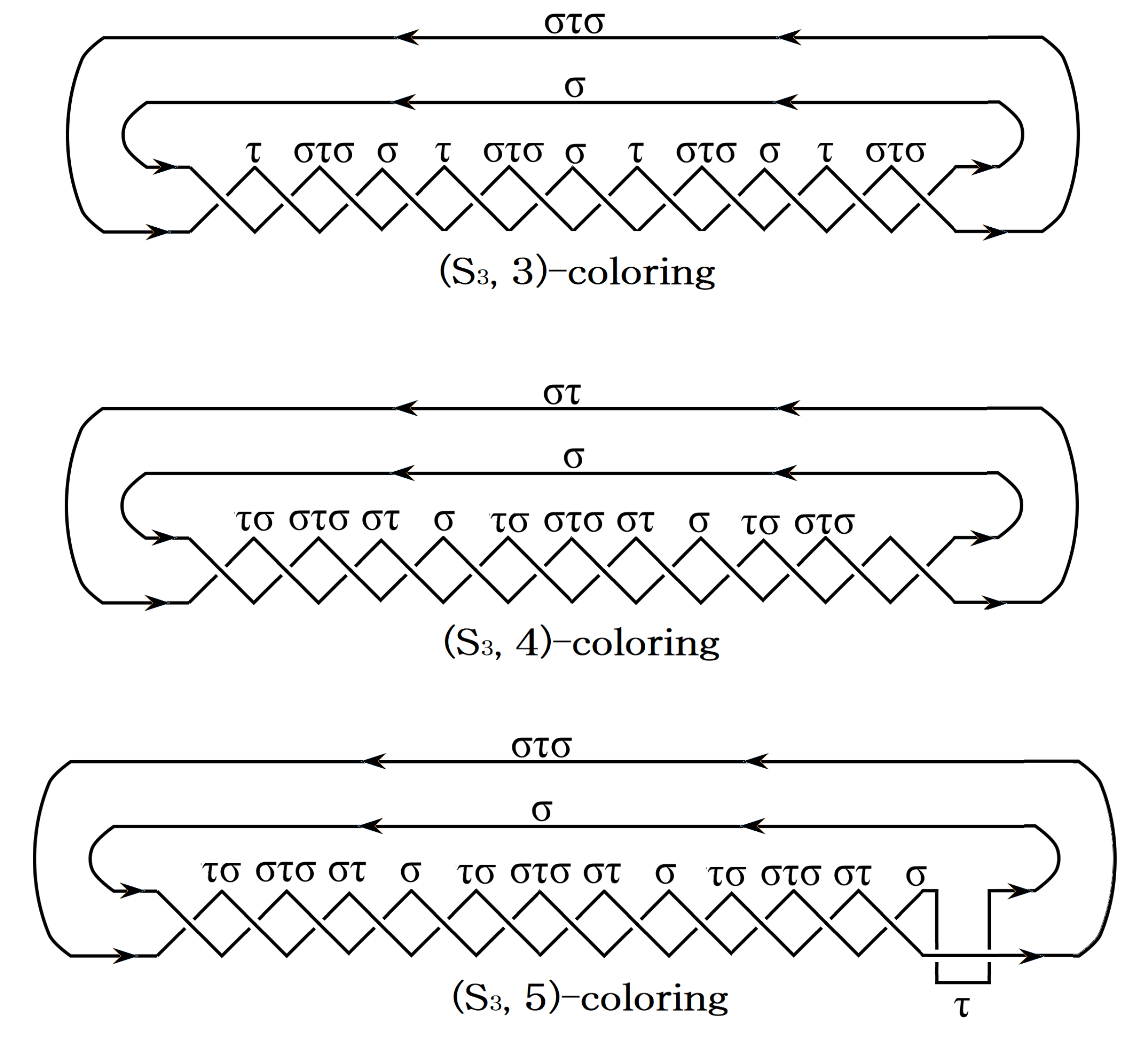}
\caption{$S_3$-colorings for $T(2,12)$}
\label{torus-examples.pdf}
\end{figure}
\end{example}

Next, we consider \textit{double-twist links}, which are the links admitting the diagrams shown in Figure~\ref{2-1}. 

\begin{figure}[htb]
\centering
\includegraphics[width=.45\textwidth]{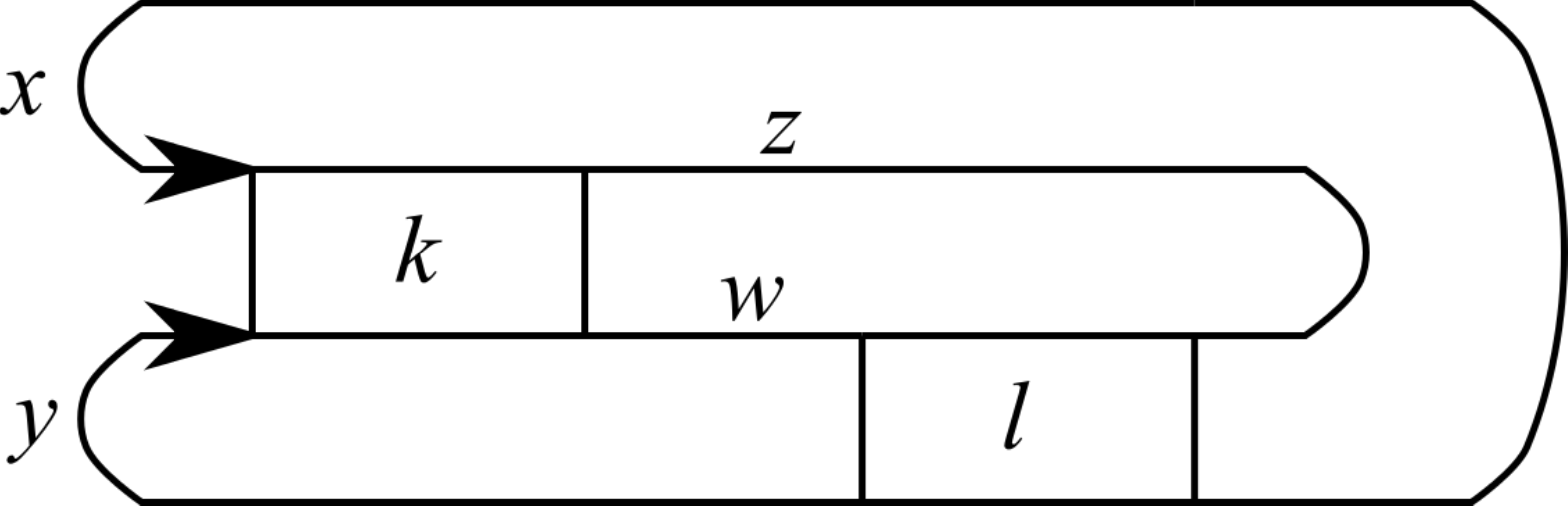}
\caption{a diagram of a double twist link $J(k,l)$}
\label{2-1}
\end{figure}

An example of the double twist link with $(S_3,4)$-coloring which is not $(S_3,3)$-colorable is depicted in Figure~\ref{J35}. 

\begin{figure}[htb]
\centering
\includegraphics[width=.6\textwidth]{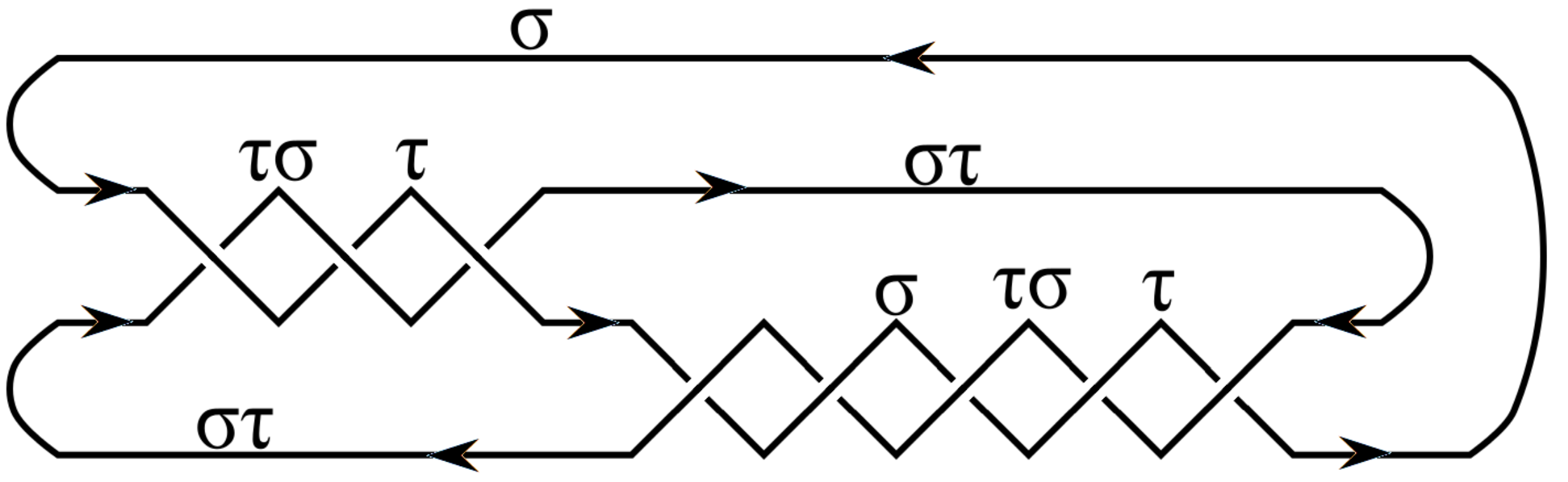}
\caption{Double twist link $J(3,5)$ with an $(S_3,4)$-coloring}
\label{J35}
\end{figure}

Actually, for double twist links, we have the following. 

\begin{proposition}
A double twist link $J(k,l)$ depicted in Figure~\ref{2-1} is $(S_3,4)$-colorable if and only if $kl\equiv3 \pmod {4}$,  and is $(S_3,3)$-colorable if and only if $kl \equiv 2\pmod {3}$.
\end{proposition}

\begin{proof}
To see which $J(k,l)$ is $(S_3,4)$-colorable, we need to consider Conway diagrams to apply Theorem~\ref{main}, but here, we directly consider the diagram $D$ of $J(k,l)$ shown in Figure~\ref{2-1}. 

First we show that $D$ is $(S_3,4)$-colorable if $kl\equiv3 \pmod {4}$. 
We set colors $\Gamma(x)$, $\Gamma(y)$ of arcs $x,y$ on Figure~\ref{2-1} as $\Gamma(x)=\sigma,\Gamma(y)=\sigma\tau$. 
Then the pair of colors $(\Gamma(z),\Gamma(w))$ on arcs $(z,w)$ is fixed as $(\tau\sigma,\sigma)$ with $k\equiv1\pmod 4$, or $(\sigma\tau,\tau)$ with $k\equiv3\pmod 4$ to make a coloring on $D$ by Table~\ref{table1}. 
For the case of $k\equiv1\pmod {4}$, $l\equiv3\pmod{4}$ also holds, and so $D$ is $S_3$-colorable as $(\Gamma(x),\Gamma(y),\Gamma(z),\Gamma(w))=(\sigma,\sigma\tau,\tau\sigma,\sigma)$. 
See Figure~\ref{2-2}. 
Note that $\sigma\tau\sigma$ does not appear during the twists, that is, the coloring is an $(S_3,4)$-coloring. 

\begin{figure}[htb]
\centering
\includegraphics[width=.45\textwidth]{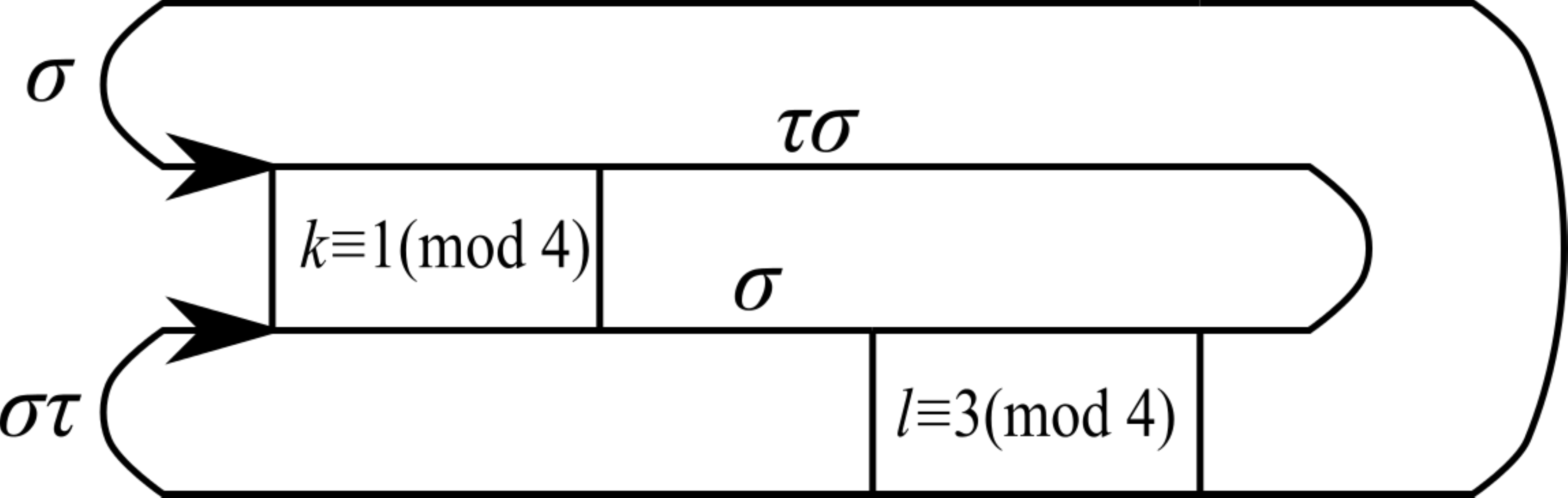}
\caption{a diagram of a double twist link $J(k,l)$}
\label{2-2}
\end{figure}

In the same way, in the case of $k\equiv3\pmod{4}$, $D$ is shown to be $(S_3,4)$-colorable. 

Conversely, suppose that $J(k,l)$ is $(S_3,4)$-colorable. 
In the same argument as the proof of Theorem~\ref{main}, the digram $D$ of $J(k,l)$ admits a $S_3$-coloring such that the arcs contained in one component are all colored by either of $\{\sigma,\tau,\sigma\tau\sigma\}$ or $\{\sigma\tau,\tau\sigma\}$. 
Then, as above, by seeing the colors on the arcs from the left end, one can check that the condition $kl\equiv3 \pmod {4}$ is necessary. 

For $(S_3,3)$-colorability, again, by \cite{KauffmanLopes}, the determinant of $J(k,l)$ is shown to be $1+kl$, and so, $J(k,l)$ is Fox 3-colorable, equivalently, is $(S_3,3)$-colorable if and only if $kl \equiv 2\pmod {3}$. 
\end{proof}

\bigskip

\noindent
\textbf{Acknowledgements.}
The authors would like to thank to Masaaki Suzuki for useful discussions. 
Also they thank to anonymous referee of the previous submission.

\end{document}